\documentclass[a4paper,11pt]{article}
\usepackage{makeidx}
\usepackage[utf8x]{inputenc}
\usepackage[T1]{fontenc}
\usepackage{lmodern}
\usepackage{textcmds}
\usepackage{epsfig}
\usepackage{amsfonts}
\usepackage{amsthm}
\usepackage{latexsym}
\usepackage{amsmath}
\usepackage{amscd}
\usepackage{amssymb}
\usepackage{enumerate}
\usepackage{textcomp}
\usepackage{yfonts}
\usepackage{setspace}
\usepackage{diagrams}
\usepackage{geometry}

\usepackage{authblk}
\usepackage{subscript}

\theoremstyle{plain}% default
\newtheorem{lm}{Lemma}[section]
\newtheorem{thm}[lm]{Theorem}
\newtheorem{cor}[lm]{Corollary}
\newtheorem{lem}[lm]{Lemma}
\newtheorem{prop}[lm]{Proposition}
\theoremstyle{definition}

\newtheorem{defi}[lm]{Definition}

\newtheorem{note}[lm]{Notation}

\theoremstyle{remark}
\newtheorem{rem}[lm]{Remark}

\newtheorem{para}[lm]{}

\newcommand{\mr}[1]{\mathrm{#1}}
\newcommand{\mc}[1]{\mathcal{#1}}

\newcommand{\mf}[1]{\mathfrak{#1}}

\newcommand{\ov}[1]{\overline{#1}}

\DeclareMathOperator{\m}{\mathcal{M}}

\DeclareMathOperator{\msp}{\mathrm{Spec}}

\DeclareMathOperator{\Hc}{\mathcal{H}om}

\DeclareMathOperator{\mo}{\mathcal{O}}

\DeclareMathOperator{\Hom}{\mathrm{Hom}}

\title{Smoothness of moduli space of stable torsion-free sheaves with fixed determinant in mixed characteristic}
\author{Inder Kaur}
\date{\today}

\begin{document}
 
\maketitle

\begin{abstract}
Let $R$ be a complete discrete valuation ring with fraction field of characteristic $0$ and algebraically closed residue field of characteristic $p>0$.
Let $X_R \to \msp(R)$ be a smooth projective morphism of relative dimension $1$.
We prove that, given a line bundle $\mc{L}_R$ the moduli space of Gieseker stable torsion-free sheaves of rank $r\geq 2$ over $X_R$, with determinant $\mc{L}_R$, is smooth over $R$.
\end{abstract}

\section{Introduction}

\begin{note}\label{n}
Let $R$ be a complete discrete valuation ring with maximal ideal $\mf{m}$.
Denote by $K$ its fraction field of characteristic $0$ and by $k$ its residue field of characteristic $p>0$. Assume $k$ is algebraically closed.
Let $X_R \to \msp(R)$ be a smooth fibred surface and $X_k$ its special fibre. 
Fix a line bundle $\mc{L}_R$ on $X_R$. 
Let $P$ be a fixed Hilbert polynomial.
Throughout this note, semistability always refers to Gieseker semistability (see \cite[Definition $1.2.4$]{HL}).
\end{note}

\noindent In \cite[Theorem $0.2$]{LA1}, Langer proves that the moduli functor of semi(stable)torsion-free sheaves with fixed Hilbert polynomial $P$ on $X_{R}$ is uniformly (universally) corepresented by an $R$-scheme $M_{X_R}(P)$ (respectively $M^{s}_{X_R}(P)$). 
Recall the definition of the moduli functor of flat families of (semi)stable torsion free sheaves with fixed Hilbert poynomial $P$ and determinant $\mc{L}_{R}$ on $X_{R}$ (see Definition \ref{mrl}).
We denote this functor by  $\m^{s}_{X_R,\mc{L}_{R}}$.
In this note we prove the following: 

\begin{thm}[see Proposition \ref{msklcorep}, Remark \ref{mklcorep} and Theorem \ref{smooth}]\label{mainthm} We have the following:
\begin{enumerate}
\item The moduli functor $\m_{X_R,\mc{L}_{R}}$ is uniformly corepresented by a projective $R$-scheme of finite type denoted $M_{R,\mc{L}_R}$.
The open subfunctor $\m^{s}_{X_R,\mc{L}_{R}}$ for stable sheaves is universally corepresented by a $R$-scheme of finite type, denoted $M^{s}_{R,\mc{L}_R}$.   
\item The morphism $M^{s}_{R,\mc{L}_R} \to \msp(R)$ is smooth.  
\end{enumerate}
\end{thm}

\noindent Part $1$ is proven analogously to \cite[Theorem $3.1$]{EL}. For part $2$, we prove that the deformation functor at a point in the moduli space $M^{s}_{R, \mc{L}_R}$ is unobstructed (see Theorem \ref{artkn}).

\vspace{0.2 cm}
\noindent Note that Theorem \ref{mainthm} is proven by Langer in the case when $R$ is a $k$-algebra (see \cite[Proposition $3.4$]{LA3}). 
However, the proof does not generalize to our setup. This is because it relies on \cite[Proposition $1$]{A}, the proof of which does not hold in mixed characteristic. 
The main difficulty is that even in the case of vector bundles it uses the structure of $R$ as a $k$-algebra in a fundamental way (see \cite[Section $3$]{A}). 
We use the same philosophy as \cite[Proposition $1$]{A} (of using Cech cohomology) but take a more direct approach since we are working on a family of curves.  
%We circumvent this problem by making use of the bijection on points in Theorem \cite[Theorem $3.1$]{EL}, %i.e. that a $k$-rational point of the moduli space $M^{s}_{X_{R}/R}$ corresponds to a vector bundle on %the special fibre $X_{k}$. 

\vspace{0,2 cm}
The setup is as follows: in \S \ref{bas} we recall the basic definitions and results needed for this note. 
We also prove the existence of the moduli space of stable torsion free sheaves with fixed determinant over $\msp(R)$. 
In \S \ref{deform} we show that the deformation functor at a point in the moduli space $M^{s}_{R,\mc{L}_R}$ is unobstructed. 
Finally in \S \ref{results} we prove that this moduli space is smooth over $\msp(R)$.

\vspace{0.2 cm}
\emph{Acknowledgements}: The author thanks Prof. A. Langer for a discussion during the conference 'Topics in characteristic $p>0$ and $p$-adic Geometry '. 
The author is grateful to the Berlin Mathematical School for financial support.

\section{Basic Definitions and results}\label{bas}

Keep Notations \ref{n}

\vspace{0.2 cm}
\noindent In this section we define the moduli functor of (semi)stable sheaves with fixed determinant. We prove that it is uniformly corepresented by an $R$-scheme of finite type. 

\vspace{0.2 cm}
\begin{defi}
Let $X_R \to \msp(R)$ be as in Notation \ref{n}.
\begin{enumerate}
\item Let $\m_{X_{R}/\msp(R)}(P)$ (as in \cite[Theorem $3.1$]{EL}) of pure Gieseker semistable sheaves. 
For simplicity we will denote this functor by $\m_R$ and the corresponding moduli space by $M_R$.
Denote by $\mc{P}ic_{X_R}$ the moduli functor for line bundles. 
By assumption $X_{R} \to \msp(R)$ is flat, projective with integral fibres, therefore by \cite[Theorem $9.4.8$]{FGA} the functor $\mc{P}ic_{X_R}$ is representable.
We denote this moduli space by $\mr{Pic}(X_R)$.   

\item By assumption $X_{R}$ is smooth over $R$. 
By \cite[Theorem $2.1.10$]{HL}, every coherent sheaf $\mc{E}$ on $X_R$ admits a locally free resolution 
\[ 0 \rightarrow \mc{E}_{n} \rightarrow \mc{E}_{n-1} \rightarrow \dots \rightarrow \mc{E}_{0} \rightarrow \mc{E} \rightarrow 0 .\]
Then $\mr{det}(\mc{E}) := \otimes \mr{det}(\mc{E}_{i})^{(-1)^{i}}$.

\vspace{0.2 cm}
Therefore we can define a natural transformation $\mr{det}: \m_R \to \mc{P}ic_{X_R}$. 
This induces a morphism between the schemes corepresenting these functors $M_{R} \to \mr{Pic}(X_R)$. 
\end{enumerate}
\end{defi}

Now we define the moduli functor for families of pure Gieseker semistable sheaves with fixed determinant. 

\begin{defi}\label{mrl}
Let $X_{R} \to \msp(R)$ be a smooth, projective morphism and $\mc{L}_{R}$ a line bundle on $X_{R}$.
For $P$ a fixed Hilbert polynomial, we define the moduli functor $\m_{X_{R},\mc{L}_{R}}(P)$, denoted $\m_{{R},\mc{L}_{R}}$ for simplicity, on $X_{R}$ of sheaves with fixed determinant $\mc{L}_{R}$.
Let $\m_{X_R,\mc{L}_{R}} : (\mr{Sch}/R)^{\circ} \rightarrow (\mr{Sets})$ be such that for an $R$-scheme $T$,

\[ \m_{X_R,\mc{L}_{R}}(T) := \left\{ \begin{array}{l}  S \mbox{- equivalence classes of families of pure Gieseker } \\
                                       \mbox{ semistable sheaves } \mc{F} \mbox{ on } X_{T}  \mbox{ with the property that }\\
                                      \mr{det}(\mc{F}) \simeq \pi^{*}_{X_{R}} \mc{L}_{R}\otimes \pi^{*}_{T}{\mc{Q}}, \mbox{ for some line bundle } \mc{Q} \mbox{ on } T 
                                    \end{array} \right \} / \sim \] 

\noindent where $\pi_{X_{R}}: X_{T} \rightarrow X_{R}$ and $\pi_T:X_T \to T$ are the natural projection maps and  $\mc{F} \sim \mc{F'}$, if and only if there exists a line bundle $\mc{L}$ on $T$, 
such that $\mc{F} \simeq \mc{F'} \otimes \pi^{*}_{T}\mc{L}$.

\vspace{0.2 cm}
\noindent We denote by $\m^{s}_{X_R,\mc{L}_{R}}$ the subfunctor for the stable sheaves. 
\end{defi}

We note that the moduli space $M^{s}_{R,\mc{L}_{R}}$ is a projective $R$-scheme. 

\begin{prop}\label{msklcorep}
The functor $\m^{s}_{R,\mc{L}_{R}}$ is universally corepresented by a $R$-scheme of finite type.
We denote this scheme by $M^{s}_{R,\mc{L}_{R}}$.
\end{prop}

\begin{proof}
We know from the proof of \cite[Theorem $3.1$]{EL}, there exists a subset of the \emph{Quot} scheme denoted $\mc{R}^{s}$, such that  $M^{s}_{R}$ is a universal categorical quotient of this subset by the action of a certain general linear group.
Let $\alpha: \mc{R}^{s} \rightarrow M^{s}_{R}$ denote this quotient. 

\vspace{0.2 cm}
The natural transformation $\m^{s}_{R} \rightarrow \mc{P}ic_{X_{R}}$ which induces the determinant morphism $\mr{det} : M^{s}_{R} \rightarrow \mr{Pic}(X_{R})$.
By composing the morphism det with $\alpha$ we obtain, a morphism $\mr{det}_{\mc{R}^{s}} : \mc{R}^{s} \rightarrow  M^{s}_{R} \rightarrow \mr{Pic}(X_{R})$.

\vspace{0.2 cm}
\noindent Let $\mc{R}^{s}_{\mc{L}_{R}} := \mr{det}_{\mc{R}^{s}}^{-1}(\mc{L}_{R})$ denote the fibre of the map $\mr{det}_{\mc{R}^{s}}$ at the point corresponding to $\mc{L}_{R}$ and let $N_{R,\mc{L}_{R}}:= \mr{det}^{-1}(\mc{L}_{R})$.
Let $M^{s}_{R,\mc{L}_{R}}$ be a universal categorical quotient of $\mc{R}^{s}_{\mc{L}_{R}}$ by $\mr{GL}(V)$.
By definition of categorical quotient, there exists a unique morphism from $\phi^{s}_{\mc{L}_R}: M^{s}_{R,\mc{L}_{R}} \rightarrow N_{R,\mc{L}_{R}}$.
Since the quotients $\mc{R}^{s}\rightarrow M^{s}_{R}$ and $\mc{R}^{s}_{\mc{L}_{R}}\rightarrow M^{s}_{R,\mc{L}_{R}}$ are $\mr{PGL}(V)$-bundles in the fppf topology (see \cite[Lemma $6.3$]{M}),
it implies $\phi^{s}_{\mc{L}_{R}}$ is an isomorphism.
Therefore, we have the following diagram,
\[\begin{diagram}
 M^{s}_{R,\mc{L}_{R}} \simeq N_{R,\mc{L}_{R}} &\rTo & M^{s}_{R}\\
     \dTo &\square &\dTo^{\det}\\
     \msp(R) &\rTo^{[\mc{L}_{R}]} & \mr{Pic}(X_{R})\\
\end{diagram}\]

\noindent Finally by \cite[Theorem $4.3.1$]{HL} we conclude that the functor $\mc{M}^{s}_{R,\mc{L}_{R}}$ is universally corepresented by the $R$-scheme $M^{s}_{R,\mc{L}_{R}}$.
\end{proof}

\begin{rem}\label{mklcorep}
Note that the  functor $\m_{R,\mc{L}_{R}}$ is corepresented by a projective $R$-scheme, denoted $M_{R,\mc{L}_{R}}$ of finite type. Recall the proof of \cite[Theorem $3.1$]{EL}.
Since $X_{R}$ is smooth,  using \cite[Theorem $2.1.10$]{HL}, we can define a morphism $\mr{det}': \mr{Quot}_{X_{R}}(\mc{H},P) \to \mr{Pic}(X_{R})$ mapping a coherent sheaf on $X_R$ to its determinant bundle. Denote by $A$ the (scheme-theoretic) intersection of $\mr{det}^{'-1}(\mc{L}_{R})$ and $Q$, where  $Q$ as in the proof of \cite[Theorem $3.1$]{EL}. Then the statement follows after replacing $Q$ by $A$ in the proof of \cite[Theorem $3.1$]{EL}.  
\end{rem}

%Let $m$ be a sufficiently large integer and let $Q$ denote the closure of the subset of the relative Quot %scheme, parametrizing quotients of $\mc{H} := \mc{O}_{X_{R}}(-m)^{P(m)}$ with Hilbert polynomial $P$ on %geometric fibers of $X_{R} \to \msp(R)$ parametrizing pure sheaves of dimension $d = \mr{deg}(P)$.

%Since $X_{R}$ is smooth,  using \cite[Theorem $2.1.10$]{HL}, we can define a morphism $\mr{det}': %\mr{Quot}_{X_{R}}(\mc{H},P) \to \mr{Pic}(X_{R})$ mapping a coherent sheaf on $X_R$ to its determinant %bundle. Denote by $A$ the (scheme-theoretic) intersection of $\mr{det}^{'-1}(\mc{L}_{R})$ and $Q$. 
%Since $Q$ is a projective $R$-scheme so is $A$. 
%Let  $\mc{R}_{\mc{L}_{R}}$ denote the open subset of $A$ consisting of those quotients $\mc{H}_{s} \to %%%\mc{F}$ for which $\mc{F}$ is Gieseker semistable on the fibre $X_{s}$, 
%for $s$ a point in $\msp(R)$ and the induced map 
%\[ H^{0}(X_{s}, \mc{H}_{s}(m)) \to H^{0}(X_{s}, \mc{F}(m)) \] 
%\noindent is an isomorphism.
%By \cite[Theorem $4.3$]{LA2}, $\mc{R}_{\mc{L}_{R}}$ is equal to the set of points of $A$ which are %semistable for the action of $\mr{GL}(P(m))$ with respect to the embedding determined by $m$.
%Let $M_{R,\mc{L}_{R}}$ denote the quotient of $\mc{R}_{\mc{L}_{R}}$ by $\mr{GL}(P(m))$.
%By \cite[Theorem $4$]{S}, such a quotient exists and is a projective scheme of finite type over $\msp(R)$. 

%\vspace{0.2 cm}
%\noindent By \cite[Lemma $4.3.1$]{HL}, the functor $\mc{M}_{R,\mc{L}_{R}}$ is  corepresented by the %scheme $M_{R,\mc{L}_{R}}$.

\section{Deformation of moduli spaces with fixed determinant}\label{deform}

Keep Notations \ref{n}. 
We have seen in the proof of Proposition \ref{msklcorep} how $M^{s}_{R,\mc{L}_R}$ can be considered as the fiber of the determinant morphism $\mr{det}: M^{s}_{R} \to \mr{Pic}(X_{R})$ over the point corresponding to $\mc{L}_R$.
Using the trace map (see Definition \ref{trmap}), we relate the obstruction theory of the deformation functor at a point in the moduli space $M^{s}_{R}$ to the obstruction theory of 
the deformation functor at a point in the moduli space $\mr{Pic}(X_{R})$.    
We use this (see Theorem \ref{artkn}) to show that the deformation functor at a point in the moduli space $M^{s}_{R,\mc{L}_R}$ is unobstructed.

\vspace{0.4 cm}
We begin by recalling some basic definitions. 

\begin{note}\label{nfd}
We denote by $\mr{Art}/R$ the category of local artinian $R$-algebras with residue field $k$.
Denote by  $X_k :=X_R \times_{\msp(R)} \msp (k)$ and  $X_A := X_R \times_{\msp(R)} \msp(A)$.
Let $[\mc{F}_k]$ denote  a closed point of $M^s_R$. As $M^s_R \to \msp(R)$ is a morphism of finite type, the closed points of the moduli space $M^s_R$ are $k$-points.
Since $k$ is algebraically closed, by \cite[Theorem $3.1$]{EL} we have a bijection 

\[ \theta(k): \mc{M}_{R}(k) \to \mr{Hom}_{R}(k, M_{R}). \]          

Therefore to a closed point of $M^s_R$ say $[\mc{F}_{k}]$, we can associate a Gieseker stable sheaf $\mc{F}_{k}$ on the curve $X_k$.
Since the curve $X_{k}$ is smooth, the torsion-free sheaf is in fact locally free. 
\end{note}

%Since $\mc{R}^{s}_{R} \twoheadrightarrow M_R^s$ is a geometric quotient, we have $\mc{R}^{s}_{R}(\msp k) \twoheadrightarrow  M_R^s (\msp k)$. 
%Therefore, $\mc{F}_{k} \in \mc{R}^{s}_{R}(\msp k)$.  

\vspace{0.2 cm}
\noindent We define a covariant functor at the point $[\mc{F}_{k}]$ in $M^s_R$.

\begin{defi}\label{defdfk}
We define the deformation functor $\mc{D}_{[\mc{F}_{k}]}:  \mr{Art}/R \to (\mr{Sets})$, such that for $A \in \mr{Art}/ R$ 
\[  \mc{D}_{[F_{k}]}(A) : = \left\{ \begin{array}{l}
                                       \mbox{ coherent  sheaves } \mc{F}_A \mbox{ with Hilbert polynomial } P \\
                                       \mbox{ on }X_A \mbox{ flat over} A \mbox{ such that its pull-back to } X_k \\
                                       \mbox{ is isomorphic to } \mc{F}_{k}.
                                       \end{array} \right\} \] 
\end{defi}

\vspace{0.2 cm}
\noindent Similarly, we define a covariant functor at the point $[\mr{det}(\mc{F}_{k})]$ of the moduli space $\mr{Pic}(X_{R})$. 

\begin{defi}\label{defdfd} 
%The determinant morphism induces a morphism $\Lambda' \to \Lambda''$.
%Hence $\Lambda''$ algebras are also $\Lambda'$ algebras and the residue field of $\Lambda'$ is a subfield of $k$.
%Let $\mr{Art}/ \Lambda'$ be the category of local Artinian $\hat{\mo}_{\mr{Pic}(X_R),[\mr{det}(\mc{F}_{k})]}$-algebras with residue field $k$. 
Let $\mc{D}_{[\mr{det}(F_{k})]}:\mr{Art}/R \to (\mr{Sets})$ be a covariant functor such that for $A \in \mr{Art}/R$
 \[ \mc{D}_{[\mr{det}(F_{k})]}(A) := \left\{ \begin{array}{l}
                                       \mbox{ coherent  sheaves } \mc{F}_A \mbox{ with Hilbert polynomial the } \\
                                       \mbox{ same as }\mr{det}(\mc{F}_k) \mbox{ on } X_A \mbox{ flat over} A \mbox{ such that its } \\
                                      \mbox{ pull-back to }  X_{k} \mbox{ is isomorphic to } \mr{det}(\mc{F}_k).
                                       \end{array} \right\} \]                                          
\end{defi}

\vspace{0.2 cm}
The following theorem gives the obstruction theories of $\mc{D}_{[F_{k}]}$ and $\mc{D}_{[\mr{det}(F_{k})]}$.

Using this we prove the following corollary.

\begin{rem}\label{unobs1}
By {\cite[Theorem $7.3$]{HD}}] the functors $\mc{D}_{[\mc{F}_{k}]}$ and $\mc{D}_{[\mr{det}(F_{k})]}$ have obstruction theories in the groups 
$H^2(\Hc_{X_{k}}(\mc{F}_{k},\mc{F}_{k})\otimes_k I)$ and $H^2(\Hc_{X_{k}}(\det (\mc{F}_{k}),\det (\mc{F}_{k}))\otimes_k I)$ respectively.
For $X_{k}$ a curve, by Grothendieck vanishing theorem, $H^2(\Hc_{X_{k}}(\mc{F}_{k},\mc{F}_{k})\otimes_k I)$ and $H^2(\Hc_{X_{k}}(\det (\mc{F}_{k}),\det (\mc{F}_{k}))\otimes_k I)$ vanish.
Therefore, $\mc{D}_{[\mc{F}_{k}]}$ and $\mc{D}_{[\mr{det}(F_{k})]}$ are unobstructed.
\end{rem}

\vspace{0.2 cm}
\noindent Now we define a natural transformation between the two deformation functors.

\begin{defi}\label{ntdet}
By assumption $\mc{F}_{k}$ is a locally-free $\mc{O}_{X_{R}}$ module.
Moreover, by \cite[Exercise $7.1$]{HD} any coherent sheaf $\mc{F}_{A}$ on $X_{A}$ satisfying the property $\mc{F}_{A}\otimes_{\mc{O}_{X_{A}}}\mc{O}_{X_{k}}\simeq \mc{F}_{k}$ is a locally free $\mc{O}_{X_{A}}$-module. Therefore, the notion of determinant is well-defined for any coherent sheaf on $X_{A}$ which pulls back to $\mc{F}_{k}$.

\vspace{0.2 cm}
\noindent We define a natural transformation $\mc{D}et: \mc{D}_{[\mc{F}_{k}]}\to \mc{D}_{[\mr{det}(F_{k})]}$ such that for  $A \in \mr{Art}/ R$,
\[ \mc{D}et_{A} :\mc{D}_{[F_{k}]}(A) \to \mc{D}_{[\mr{det}(F_{k})]}(A),  \ \ \mc{E}_{A} \mapsto \mr{det}(\mc{E}_{A}). \]
\end{defi}

\vspace{0.2 cm}
\noindent Using this we define a deformation functor at a point in the moduli space $M^{s}_{R, \mc{L}_{R}}$.

\begin{defi}\label{defdfl} 
Let $\mc{L}_R$ be as in Notation \ref{n}. 
For $A$ a $R$-algebra, denote by $\mc{L}_{A}$ the pullback $p_{A}^{*}{\mc{L}_{R}}$ under the natural morphism $p_A: X_A \to X_{R}$.

\vspace{0.2 cm}
\noindent We define a functor $\mc{D}_{[\mc{F}_{k}],[\mr{det} \mc{F}_{k}]} :\mr{Art}/R \to (\mr{Sets})$, such that for $A \in \mr{Art}/R$. 
 \[ \mc{D}_{[\mc{F}_{k}],[\mr{det} \mc{F}_{k}]}(A) := \mc{D}et_{A}^{-1}(\mc{L}_{A}).\]   
\end{defi}

%\begin{rem}
%Our goal in this section is to prove that $\mc{D}_{[\mc{F}_{k}],[\mr{det}(\mc{F}_{k})]}$ is unobstructed.
%For this, we relate the obstruction theory of the functor $\mc{D}_{[\mc{F}_{k}],[\mr{det}(\mc{F}_{k})]}$ to the obstruction theories of the functors $\mc{D}_{[F_{k}]}$ and $\mc{D}_{[\mr{det}(F_{k})]}$.
%By Theorem  the set of extensions $\mc{F}_{A'}$ of $\mc{F}_{A}$ over $X_{A'}$ is a torsor under the action of $H^1(\Hc_{X_{k}}(\mc{F}_{k},\mc{F}_{k})\otimes_k I)$.
%Now we try to understand the action of the group $H^1(\Hc_{X_{k}}(\mc{F}_{k},\mc{F}_{k})\otimes_k I)$ on an extension $\mc{F}_{A'}$ over $A'$.
%\end{rem}

\vspace{0.2 cm}

\begin{para}\label{b4glue}
\emph{Group action on the torsors}: By \cite[Theorem $7.3$]{HD}, the set $\mc{D}_{[F_{k}]}(A')$ (respectively $\mc{D}_{[\mr{det}(F_{k})]}(A')$) is a torsor under the action of $H^1(\Hc_{X_{k}}(\mc{F}_{k},\mc{F}_{k})\otimes_k I)$ (respectively $H^1(\Hc_{X_{k}}(\det (\mc{F}_{k}),\det (\mc{F}_{k}))\otimes_k I)$).
%Now we recall the action of these groups on the respective torsors as given in the proof of \cite[Theorem $7.3$]{HD}.

\vspace{0.2 cm}
Since $X_k$ is noetherian, we can identify the sheaf cohomology $H^{1}(X_{k},\Hc(\mc{F}_{k},\mc{F}_{k})\otimes_{k} I)$ with the Cech cohomology $\check{H}^1(\mc{U},\Hc(\mc{F}_{k},\mc{F}_{k})\otimes_{k} I)$, 
where $\mc{U}$ is an affine open covering of $X_{k}$.
Then an element, say $\xi$ of the cohomology group $H^{1}(\Hc(\mc{F}_{k},\mc{F}_{k})\otimes_{k} I)$ can be seen as a collection of elements $\{\phi'_{ij}\} \in \Gamma(U_i \cap U_j, \Hc(\mc{F}'_{k},\mc{F}'_{k}))$ 
satisfying the cocycle condition i.e. for any $i, j,k$, we have $\phi'_{ik}|_{U_{ijk}} = \phi'_{jk}|_{U_{ijk}} + \phi'_{ij}|_{U_{ijk}}$.
Since $I$ is a $k$-vector space, $\check{H}^1(\mc{U},\Hc(\mc{F}_{k},\mc{F}_{k})\otimes_{k} I) \simeq \check{H}^1(\mc{U},\Hc(\mc{F}_{k},\mc{F}_{k}))\otimes_{k} I$.
Therefore, $\{\phi'_{ij}\}_{i,j}$ is of the form $\{\phi''_{ij}\otimes a\}_{i,j}$ for $a \in I$ not depending on $i,j$ and $\phi''_{ij} \in \Gamma(U_i \cap U_j,\Hc(\mc{F}_{k},\mc{F}_{k}))$ satisfying 
$\phi''_{ik}|_{U_{ijk}} = \phi''_{jk}|_{U_{ijk}} + \phi''_{ij}|_{U_{ijk}}$.

\vspace{0.2 cm}
%Lemma \ref{glue} implies that using the collection $\{\phi'_{ij}\}$ of Definition \ref{b4glue} we can modify the glueing data of $\mc{F}_{A'}$.
%To this end we produce a collection of sheaves as in Lemma \ref{glue} and isomorphisms which agree on intersection with $X_{k}$ to obtain an extension $\mc{F'}_{A'}$ from $\mc{F}_{A'}$ of $\mc{F}_{A}$. 

Let $\mc{F}_{A'}$ be an extension of $\mc{F}_{A}$ on $X_{A'}$ i.e an element of $\mc{D}_{[\mc{F}_{k}]}(A')$. Since it is locally free, there exists a covering $\mc{U}' = \{U'_{i}\}$ of $X_{A'}$ by such that $\mc{F}_{A'}|_{U'_{i}}$ is $\mo_{U'_{i}}$-free.
% We now describe the action of the element $\xi$ on $\mc{F}_{A'}$ which we denote by $\mc{F}_{A'}(\xi)$.
Denote by $\mc{U}:= \{U_{i}\}$ the cover of $X_{k}$ where $U_{i}:= U'_{i} \cap X_{k}$. 
We know from the proof of {\cite[Theorem $7.3$]{HD}} that $\mc{F}_{A'}(\xi)$ is given by a collection of sheaves $\mc{F}'_i:= \mc{F}_{A'}|_{U'_i}$ and isomorphisms $\phi_{ij}:\mc{F}'_i|_{U_i \cap U_j} \to \mc{F}'_j|_{U_i \cap U_j}$ such that 
\[\phi_{ii}=\mr{Id}, \, \, \phi_{ij}:\mc{F}'_i|_{U'_i \cap U'_j}=\mc{F}_{A'}|_{U'_i \cap U'_j} \xrightarrow{\mr{Id}+((\phi_{ij}'' \otimes a) \circ \pi)} \mc{F}_{A'}|_{U'_i \cap U'_j} =\mc{F}'_j|_{U'_i \cap U'_j}\] 
\noindent where $\phi_{ij}'', a$ 
are as above and $\pi$ is the natural restriction morphism $\mc{F}_{A'} \to \mc{F}_{k}$.
Then by \cite[Ex. II.$1.22$]{H},  $\mc{F}_{A'}(\xi)$ glues to a sheaf if the morphisms $\{\phi_{ij}\}$ satisfy the cocycle condition.
In the following lemma we prove that this is indeed the case.
\end{para}

%\begin{lem}\label{glue}
%Given a collection $\{\mc{F}_i\}$ of sheaves with $\mc{F}_i$ a sheaf on $U_i$ and isomorphisms $\phi_{ij}:\mc{F}_i|_{U_i \cap U_j} \to \mc{F}_j|_{U_i \cap U_j}$ satisfying
%$\phi_{ii}=\mr{Id}$ and the \emph{cocycle condition} $\phi_{ik}=\phi_{jk} \circ \phi_{ij}$ on $U_i \cap U_j \cap U_k$ the sheaves $\mc{F}_i$ ``glue uniquely'' i.e., there exists an unique sheaf $\mc{F}$ on $X_{A'}$,
%together with isomorphisms $\psi_i:\mc{F}|_{U_i} \to \mc{F}_i$
%such that for each $i,j$ we have,  $\psi_j=\phi_{ij} \circ \psi_i$ on $U_i \cap U_j$. 
%\end{lem}

\begin{lem}\label{actgluelem}
Let $\mc{F}'_{i}$ and $\phi_{ij}$ be as above.
The morphisms $\{\phi_{ij}\}$ satisfy the cocycle condition i.e. for any $i,j,k  \ \phi_{ik} = \phi_{jk}\circ \phi_{ij}$.
\end{lem}

\begin{proof}
It suffices to prove this equality for the basis elements, say $s^{i}_{1},\dots, s^{i}_{r}$ generating $\mc{F}'_{i}|_{U'_{i}\cap U'_{j} \cap U'_{k}}$.
For any basis element $s^i_{t}$,                                 
\begin{eqnarray*}
\phi_{jk} \circ \phi_{ij}(s^i_{t}) & = & \phi_{jk}(\mr{Id} + (a \otimes \phi''_{ij}))(\pi(s^i_{t}))\\
& = & \phi_{jk}(\pi(s^i_{t}) + a\phi''_{ij}(\pi(s^i_{t}))) = (\mr{Id} + a \otimes \phi''_{jk})(\pi(s^i_t))+ a\phi''_{ij}(\pi(s_{t}^{i}))\\
& = & \pi(s^i_t) + a \phi''_{ij}(\pi({s}_{t}^i)) +  a \phi''_{jk}(\pi(s_{t}^{i})) + 0
\end{eqnarray*}

\noindent because $a^{2}= 0$ in $A'$.
Since $\phi''_{ik} = \phi''_{ij} + \phi''_{jk}$,  we have 
\[\phi_{jk}\circ \phi_{ij} (s_{t}^{i}) = \pi(s_{t}^{i})+ a(\phi''_{ik}(\pi(s_{t}^{i})) = \phi_{ik}(s_{t}^{i}). \]
This shows that $\{\phi_{ij}\}_{i,j}$ satisfy the cocycle condition.
\end{proof}

\noindent Using this we conclude that $\mc{F}_{A'}(\xi)$, obtained by glueing the sheaves $\mc{F}'_i$ along the isomorphism $\phi_{ij}$ is a sheaf.

\vspace{0.2 cm}
Similarly,  an element say $\xi'$ in $H^1(\Hc_{X_{k}}(\det (\mc{F}_{k}),\det (\mc{F}_{k}))\otimes_k I)$ acts on an element in $\mc{D}_{[\mr{det}(\mc{F}_k)]}(A')$, 
say $\mr{det}(\mc{F}_{A'})$ to produce a line bundle $\mr{det}(\mc{F}_{A'})(\xi')$ given by a family of sheaves $\{\mc{L}_{i} := \mc{L}_{A'}|_{U_{i}}\}$ and isomorphisms 
\[\phi_{ij}:\mc{L}_{i}|_{U'_{ij}} \xrightarrow{\mr{Id}+((\phi_{ij}'' \otimes a) \circ \pi)} \mc{L}_{j}|_{U'_{ij}}\] 
where $\phi''_{ij} \in \Gamma(U_i \cap U_j,\Hc(\mr{det}(\mc{F}_{k}),\mr{det}(\mc{F}_{k}))\otimes_k I)$ is the collection of sections corresponding to $\xi'$ 
given by the ismomorphism $H^1(\Hc_{X_{k}}(\det (\mc{F}_{k}),\det (\mc{F}_{k}))\otimes_k I) \simeq \check{H}^1(\mc{U},\Hc(\mr{det}(\mc{F}_{k}),\mr{det}(\mc{F}_{k})))\otimes_{k} I$.
Again by Lemma \ref{actgluelem}, $\mr{det}(\mc{F}_{A'})(\xi')$ is a sheaf.

\begin{defi}\label{d10} 
We have the following definitions.
\begin{enumerate}
\item  We define  a map  
\[\phi_1:H^1(\Hc_{X_{k}}(\mc{F}_{k},\mc{F}_{k})\otimes_k I) \to \mc{D}_{[\mc{F}_{k}]}(A'),   \ \ \  \xi \mapsto  \mc{F}_{A'}(\xi)\]

\noindent which uniquely associates an extension $\mc{F}_{A'}(\xi)$ of $\mc{F}_{A'}$ (using Lemma \ref{actgluelem}) to an element $\xi$ of $H^1(\Hc_{X_{k}}(\mc{F}_{k},\mc{F}_{k})\otimes_k I)$.

\item  Replacing $\mc{F}_{A'}$ by $\mr{det}(\mc{F}_{A'})$ and starting with $\mr{det}(\mc{F}_{A'})$ we associate an extension say $(\mr{det}(\mc{F}_A'))(\xi')$ 
to an element $\xi'$ of $H^1(\Hc_{X_{k}}(\det (\mc{F}_{k}),\det(\mc{F}_{k}))\otimes_k I)$.
Hence we define a map
\[\phi_2:H^1(\Hc_{X_{k}}(\det (\mc{F}_{k}),\det (\mc{F}_{k}))\otimes_k I) \to \mc{D}_{\det (\mc{F}_{k})}(A'), \ \ \  \xi' \mapsto  \mr{det}(\mc{F}_A')(\xi') \]
\end{enumerate}
\end{defi}

\begin{rem}\label{defr}
Note that by Corollary \ref{unobs1}, there exist surjective morphisms $r_1: \mc{D}_{\mc{F}_{k}}(A') \twoheadrightarrow \mc{D}_{\mc{F}_{k}}(A)$ and $r_2: \mc{D}_{\det (\mc{F}_{k})}(A') \twoheadrightarrow \mc{D}_{\det (\mc{F}_{k})}(A)$. 
By \cite[Theorem $7.3$]{HD}, $r_1^{-1}(\mc{F}_{A}) = \mr{Im}(\phi_{1})$, $r_2^{-1}(\mr{det}(\mc{F}_{A})) = \mr{Im}(\phi_{2})$.
\end{rem}

\vspace{0.2 cm}
The following lemma tells us that taking the determinant commutes with glueing of the sheaf.  

\begin{lem}\label{detglue}
The determinant of the sheaf $\mc{F}_{A'}(\xi)$ is the line bundle obtained by glueing $\{\det (\mc{F}'_i)\}$ along the isomorphisms 
\[\ov{\phi}_{ij}:\det (\mc{F}'_i)|_{U_i \cap U_j} \to \det (\mc{F}'_j)|_{U_i \cap U_j}, \, \, s_1^{(i)} \wedge ... \wedge s_r^{(i)} \mapsto  \phi_{ij}(s_1^{i}) \wedge ... \wedge \phi_{ij}(s_r^{i})  \]
\noindent where $s_1^{i},...,s_r^{i}$ are the basis elements of $\mc{F}'_i|_{U_i \cap U_j}$.
    \end{lem}

\begin{proof}
By Lemma \ref{actgluelem} for all $t = 1,\dots, r$, we have  $\phi_{ik}(s_{t}^{i}) = \phi_{jk}(s^i_{t}) \circ \phi_{ij}(s_{t}^{i})$.
Then, 
\begin{eqnarray*}
 \ov{\phi}_{jk} \circ \ov{\phi}_{ij}(s_{1}^i \wedge \dots \wedge s^{i}_{r}) & = & \ov{\phi}_{jk}(\phi_{ij}(s_{1}^i) \wedge \dots \wedge \phi_{ij}(s_{r}^i))\\
 & = & (\phi_{jk} \circ \phi_{ij}(s_{1}^{i})) \wedge \dots \wedge (\phi_{jk} \circ \phi_{ij}(s_{r}^{i})) \\
 & = & \phi_{ik}(s_{1}) \wedge \dots \wedge \phi_{ik}(s_{r}^{i}) \\
 & = & \ov{\phi}_{ik}(s_{1}^i \wedge \dots \wedge s_{r}^{i}) \\
\end{eqnarray*}

\noindent Hence the morphisms $\{\ov{\phi}_{ij}\}$ satisfy the cocycle condition i.e $\ov{\phi}_{ik} = \ov{\phi}_{jk} \circ \ov{\phi}_{jk}$. 

\vspace{0.2 cm} By Lemma \ref{actgluelem}, there exist isomorphisms $\psi_{i}: \mc{F}_{A'}(\xi)|_{U'_{i}} \simeq \mc{F}'_{i}$ satisfying $\psi_{j}|_{U_{ij}} = \phi_{ij} \circ \psi_{i}|_{U_{ij}}$.  
We define $\ov{\psi}_{i}: \mr{det}(\mc{F}_{A'}(\xi))|_{U'_{i}} \simeq \mr{det}(\mc{F}'_{i})$ as follows.
Let $s^ i_{1}, \dots , s_{r}^i$ be the basis of $\mc{F}_{A'}(\xi)|_{U_{i}}$.
Then $\ov{\psi}_{i}(s_{1}^{i} \wedge \dots \wedge s_{r}^{i}) := \psi_{i}(s_{1}^{i}) \wedge \dots \wedge \psi_{i}(s_{r}^{i})$. Therefore
 
\begin{eqnarray*}
\ov{\phi}_{ij} \circ \ov{\psi}_{i} (s_{1}^{i} \wedge \dots \wedge s_{r}^{i})  & = & \ov{\phi}_{ij} (\psi_{i}(s_{1}^{i}) \wedge \dots \wedge \psi_{i} (s_{r}^{i}) ) \\ 
& = & \phi_{ij} (\psi_{i} (s_{1}^{i} )) \wedge \dots \wedge \phi_{ij} (\psi_{i} (s^i_r) ) \\
& = & \psi_{j} (s_{1}^{i}) \wedge \dots \wedge \psi_{j} (s_{r}^{i})\\
& = & \ov{\psi}_{j} (s_{1} \wedge \dots \wedge s_{r}^{i})
\end{eqnarray*}

\noindent Then by the uniqueness of glueing mentioned in \cite[Ex. II.$1.22$]{H}, $\{ \mr{det}(\mc{F}_{i})\}$ glues along the isomorphisms $\{\ov{\phi}_{ij}\}_{i,j}$ to $\mr{det}(\mc{F}_{A'})(\xi)$. 
\end{proof}
 
\vspace{0.4 cm}
\begin{para} We relate the obstruction theory of $\mc{D}_{[F_k]}$ to that of $\mc{D}_{[\mr{det}(F_{k})]}$ by relating the action of the group $H^1(\Hc_{X_{k}}(\mc{F}_{k},\mc{F}_{k})\otimes_k I)$ on the vector bundle to the action of the group $H^1(\Hc_{X_{k}}(\det (\mc{F}_{k}),\det (\mc{F}_{k}))\otimes_k I)$ on the determinant of the vector bundle.
This relation is given by the trace map which we recall here.
\end{para}

% We do so by defining a morphism \[H^1(X_{k}, \Hc_{X_{k}}(\mc{F}_{k},\mc{F}_{k})) \to H^1(X_{k}, \Hc_{X_{k}}(\mr{det} \mc{F}_{k},\mr{det} \mc{F}_{k})).\]

% Let  $\mc{U}':=\{U'_i\}$ be an affine open cover of $X_{k}$.

\begin{defi}\label{trmap}
Let $U$ be an affine open set on which $\mc{F}_{k}$ is free, generated by sections say $s_1,...,s_r$ (for $r=\mr{rk}(F_{k})$). Recall the map,
 
   \[\mr{tr}_U:\mc{H}om(\mc{F}_{k},\mc{F}_{k})(U) \to \mc{H}om(\det (\mc{F}_{k}),\det (\mc{F}_{k}))(U), \ \ \  \ (\ast) \]
\[\phi \mapsto \mr{tr}_U(\phi):= ( s_1 \wedge ... \wedge s_r \mapsto \sum_j s_1 \wedge ..\wedge \phi(s_j) \wedge ... \wedge s_r ).\]

\noindent Let $\mc{U}:=\{U_i\}$ be a small enough open cover of $X_k$ such that $\mc{F}_{k}$ is free on each $U_i$. Then the \emph{trace map} is given by 
 
\[\mr{tr}: \Hc_{X_{k}}(\mc{F}_{k},\mc{F}_{k}) \to \Hc_{X_{k}}(\det (\mc{F}_{k}),\det (\mc{F}_{k})) \]

\noindent such that  $\mr{tr}|_{U_i} = \mr{tr}_{U_{i}}$ for any affine open set $U_i$ of $X_k$.
 
\end{defi}

\vspace{0.2 cm}  
\begin{rem}\label{remontr}
 Note that the morphism $\mr{tr}_U$ is $\mc{O}_{X_k}$ linear.
Let $f \in \mc{O}_{X_k}(U)$.
Then 
\begin{eqnarray*}
 \mr{tr}_U(f\phi) & = & s_1 \wedge ... \wedge s_r \mapsto \sum_j s_1 \wedge ..\wedge f\phi(s_j) \wedge ... \wedge s_r \\
                   & = & \sum_j f(s_1 \wedge ..\wedge \phi(s_j) \wedge ... \wedge s_r ) \\
                   & = & f \sum_j s_1 \wedge ..\wedge \phi(s_j) \wedge ... \wedge s_r \\
                   & = & f\mr{tr}_{U}(\phi).
\end{eqnarray*}
%$f \phi \mapsto s_1 \wedge ... \wedge s_r \mapsto \sum_j s_1 \wedge ..\wedge f\phi(s_j) \wedge ... \wedge s_r =  \sum_j f(s_1 \wedge ..\wedge \phi(s_j) \wedge ... \wedge s_r ) = f \sum_j s_1 \wedge ..\wedge \phi(s_j) \wedge ... \wedge s_r
%= f\mr{Im}(\phi)$.
%\item The trace morphism $\mr{tr}_U$ is close to the trace map given in \cite{A}. In particular, if we %choose a basis of sections $s_1,...,s_r$ of $\mc{F}_{k}(U)$ 
%such that $\phi(s_i)=c_is_i$ for some $c_i\in \mo_{X_{k}}(U)$ then the
%morphism $\mr{tr}_U(\phi) \in \Hom(\det (\mc{F}_{k})(U),\det (\mc{F}_{k})(U))$ maps $s_1 \wedge ... \wedge %s_r$ to $(c_1+...+c_r)s_1 \wedge ... \wedge s_r$.

\end{rem}

\begin{lem}\label{trsurj}
The morphism $\mr{tr}$ is surjective.
\end{lem}

\begin{proof}
It suffices to prove surjectivity on the level of stalks.
Let $x \in X_{k}$ be a closed point. 
Consider the induced morphism \[\mr{tr}_x:\Hc_{X_{k}}(\mc{F}_{k,x},\mc{F}_{k,x}) \to \Hc_{X_{k}}(\det(\mc{F}_{k,x}),\det (\mc{F}_{k,x}))\] 
and basis $s_1,...,s_r \in \mc{F}_{k,x}$. 
Since the map $\mr{tr}_x$ is $\mo_{X_{k},x}$ linear and $\Hc_{\mo_{X_{k}}}(\det (\mc{F}_{k,x}),\det (\mc{F}_{k,x})) \cong \mo_{X_{k},x}$, it suffices to show that $\mr{Id} \in \mr{Im}(\mr{tr}_x)$.
Let $\phi \in \Hc_{X_{k}}(\mc{F}_{k,x},\mc{F}_{k,x})$ defined as $\phi(s_{i}) = s_{i}$ for $i = 1$ and $0$ otherwise.
This concludes the proof.
\end{proof}

%The morphisms $\mr{tr}_{U_i}$ define naturally a corresponding morphism of complexes,
  % \begin{equation}\label{d1}
   %\begin{diagram}
      %\mc{C}^0(\mc{U},\mc{H}om(\mc{F}_{k},\mc{F}_{k}))&\rTo^d&\mc{C}^1(\mc{U},\mc{H}om(\mc{F}_{k},\mc{F}_{k}))&\rTo^d&...\\
      %\dTo^{\{\mr{tr}_{U_i}\}}&\circlearrowleft&\dTo^{\{\mr{tr}_{U_{ij}}\}}& \\
      %\mc{C}^0(\mc{U},\mc{H}om(\det (\mc{F}_{k}),\det (\mc{F}_{k})))&\rTo^d&\mc{C}^1(\mc{U},\mc{H}om(\det (\mc{F}_{k}),\det (\mc{F}_{k})))&\rTo^d&...
        %   \end{diagram}
          % \end{equation}

We can define the trace map cohomologically as follows: 
 
\begin{defi}
Let  $\mc{U}:=\{U_i\}$ be a small enough open affine cover of $X_{k}$ such that $\mc{F}_{k}$ is free on each $U_i$. Using \cite[III. Theorem $4.5$]{H} we define \v{C}ech cocycle $\mc{C}^p(\mc{U},\mc{H}om(\mc{F}_{k},\mc{F}_{k}))$ (resp $\mc{C}^p(\mc{U},\mc{H}om(\det (\mc{F}_{k}),\det (\mc{F}_{k}))$),
such that the corresponding \v{C}ech cohomology coincides with the sheaf cohomology $H^i(X_{k},\mc{H}om(\mc{F}_{k},\mc{F}_{k}))$ (resp $H^i(X_{k},\mc{H}om(\det (\mc{F}_{k}),\det (\mc{F}_{k})))$).
\noindent The morphism $(\ast)$ of Definition induces a morphism on cohomologies
 \[\mr{tr}^i:H^i(X_{k},\mc{H}om(\mc{F}_{k},\mc{F}_{k})) \to H^i(X_{k},\mc{H}om(\det(\mc{F}_{k}),\det(\mc{F}_{k}))) \cong H^i(X_{k},\mc{O}_{X_{k}}).\]
\end{defi}

As a corollary to Lemma \ref{trsurj} we have:

\vspace{0.2 cm}
\begin{cor}\label{d11}
 The morphism induced on cohomology \[\mr{tr}^1:H^1(X_{k}, \Hc_{X_{k}}(\mc{F}_{k},\mc{F}_{k})) \to  H^1(X_{k}, \Hc_{X_{k}}(\det (\mc{F}_{k}),\det (\mc{F}_{k}))\] is surjective.
 \end{cor}

 \begin{proof}
Consider the short exact sequence, \[0 \to \ker \mr{tr} \to \Hc_{X_{k}}(\mc{F}_{k},\mc{F}_{k}) \xrightarrow{\mr{tr}} \Hc_{X_{k}}(\det (\mc{F}_{k}),\det (\mc{F}_{k})) \to 0.\]
We get the following terms in the associated long exact sequence, \[... \to H^1(X_{k}, \Hc_{X_{k}}(\mc{F}_{k},\mc{F}_{k})) \xrightarrow{\mr{tr}^1}  H^1(X_{k}, \Hc_{X_{k}}(\det (\mc{F}_{k}),\det (\mc{F}_{k}))) \to H^2(\ker \mr{tr}) \to ...\]
Since $X_k$ is a curve, by Grothendieck's vanishing theorem, $H^2(\ker(\mr{tr}))= 0$. Therefore, the morphism $\mr{tr}^1$ is surjective.
  \end{proof}

The following proposition tells us that the determinant map 'commutes' with the trace map. 

\begin{prop}\label{d6}
Notation as in \ref{b4glue}.
Let \[\mr{det}_{ij}: \Gamma(U_i \cap U_j, \Hc(\mc{F}'_i,\mc{F}'_i))\to \Gamma(U_i \cap U_j, \Hc(\mr{det} (\mc{F}'_i),\mr{det} (\mc{F}'_i)))\]  be a morphism defined by 
 \[\phi_{ij} \in \Gamma(U_i \cap U_j, \Hc(\mc{F}'_i,\mc{F}'_i)) \mapsto \mr{det}_{ij}(\phi_{ij}) := (s_1^{i} \wedge ... \wedge s_r^{i} \mapsto \phi_{ij}(s_1^{i}) \wedge ... \wedge \phi_{ij}(s_r^{i}))\] 
                                               
\noindent where $s_1^{i},...,s_r^{i}$ are the basis elements of $\mc{F}'_i|_{U_i \cap U_j}$. Then for any pair $i \neq j$, we have
\[\det_{ij} \circ (\mr{Id}+(\phi_{ij}'' \otimes a) \circ \pi)=\mr{Id}+(\mr{tr}_{U_{ij}}(\phi_{ij}'') \otimes a) \circ \pi.\]
 In other words, the following diagram is commutative:
  \[\begin{diagram}
     H^1(\Hc_{X_{k}}(\mc{F}_{k},\mc{F}_{k})\otimes_k I) &\rTo^{\phi_1} & \mc{D}_{[\mc{F}_{k}]}(A')\\
     \dTo^{\mr{tr}^1\otimes \mr{Id}} &\circlearrowleft&\dTo^{\mc{D}et_{A'}}\\
     H^1(\Hc_{X_{k}}(\det (\mc{F}_{k}),\det (\mc{F}_{k}))\otimes_k I) &\rTo^{\phi_2} &\mc{D}_{[\det (\mc{F}_{k})]}(A')\\
    \end{diagram}\]
\end{prop}
   
 \begin{proof} 
Let $s_1^{i},...,s_r^{i}$ be the sections generating $\mc{F}'_i|_{U_i \cap U_j}$.
\\ Any section of $\Hc(\det (\mc{F}'_i),\det (\mc{F}'_i))$ is (uniquely) defined by the image of $s_1^{i} \wedge ... \wedge s_r^{i}$.
Hence it suffices to prove \[(\mr{det}_{ij} \circ (\mr{Id}+(\phi_{ij}'' \otimes a) \circ \pi))(s_1^{i} \wedge ... \wedge s_r^{i})=(\mr{Id}+(\mr{tr}_{U_{ij}}(\phi_{ij}'') \otimes a) \circ \pi)(s_1^{i} \wedge ... \wedge s_r^{i}).\] 
For $1\leq t \leq r$, $(\mr{Id}+(\phi_{ij}'' \otimes a) \circ \pi)(s_t^{i})=s_t^{i}+a\phi_{ij}''(\pi(s_t^{i}))$ and since $I.m_{A'}=0, a^t=0$ for $t>1$.
Hence,
\[(\mr{det}_{ij} \circ (\mr{Id}+(\phi_{ij}'' \otimes a) \circ \pi))(s_1^{i} \wedge ... \wedge s_r^{i})=(s_1^{i}+a\phi_{ij}''(\pi(s_1^{i}))) \wedge ... \wedge (s_r^{i}+a\phi_{ij}''(\pi(s_r^{i})))= \]
\[=s_1^{i} \wedge ... \wedge s_r^{i}+a\sum_k s_1^{i} \wedge ... \wedge \phi_{ij}''(\pi(s_k^{i})) \wedge ... \wedge s_r^{i}=(\mr{Id}+(\mr{tr}_{U_{ij}}(\phi_{ij}'') \otimes a) \circ \pi)(s_1^{i} \wedge ... \wedge s_r^{i}).\]
This completes the proof of the proposition.
\end{proof}
 
\vspace{0.2 cm} 
We end this section with the following theorem.

\begin{thm}\label{artkn}
The functor $\mc{D}_{[\mc{F}_{k}],[\det (\mc{F}_{k})]}$ is unobstructed.
\end{thm}

 \begin{proof}
 Let  $A' \twoheadrightarrow A$ be a small extension in $\mr{Art}/R$ and $\phi_{1},\phi_{2}$ be as in Definition \ref{d10}.
 Recall the surjective morphisms $r_1,r_2 $ from Remark \ref{defr}. 
 Then we have the following diagram.
   \[\begin{diagram}
          & &\mc{D}_{[\mc{F}_{k}],[\det (\mc{F}_{k})]}(A')&\rTo^{\psi}&\mc{D}_{[\mc{F}_{k}],[\det (\mc{F}_{k})]}(A)\\
          & &\dInto&\circlearrowleft&\dInto\\
         H^1(\Hc_{X_{k}}(\mc{F}_{k},\mc{F}_{k})\otimes_k I) &\rTo^{\phi_1} &\mc{D}_{[\mc{F}_{k}]}(A')&\rTo^r&\mc{D}_{[\mc{F}_{k}]}(A)\\
     \dTo^{\mr{tr}^1\otimes \mr{Id}} &\circlearrowleft&\dTo^{\mc{D}et_A'}&\circlearrowleft&\dTo^{\mc{D}et_A}\\
     H^1(\Hc_{X_{k}}( (\mc{F}_{k}),\det (\mc{F}_{k}))\otimes_k I) &\rTo^{\phi_2} &\mc{D}_{(\mc{F}_{k})}(A')&\rTo^r&\mc{D}_{[\det (\mc{F}_{k})]}(A)\\ 
           \end{diagram}\] 
                    
\noindent where the upper right square and the lower right square are commutative by definition and the lower left square is commutative by Proposition \ref{d6}.
To prove that $\mc{D}_{[\mc{F}_{k}],[\det (\mc{F}_{k})]}$ is unobstructed, we  need to show that $\psi$ is surjective.                    
Let $\mc{L}_A$ be the unique pull-back of $\mc{L}_R$ under the morphism $X_A \to X_R$ and $\mc{F}_A$ be an element in  $\mc{D}_{[F_{k}],[\det (\mc{F}_{k})]}(A)$. 
Since $\mc{D}_{[\mc{F}_{k}],[\det (\mc{F}_{k})]}(A') = \mr{det}(\mc{L_A'})$ where $\mc{L}_{A'}$ is $\pi^{*}\mc{L}_R$ for $\pi: X_{A'} \to X_{R}$, we need to prove there exists a sheaf
$\mc{F}_{A'}$ on $X_{A'}$ with determinant $\mc{L}_{A'}$ which is an extension of $\mc{F}_A$.

\vspace{0.2 cm}
By definition $r_2(\mc{L}_{A'})= \mc{L}_A$.
Since $\phi_1$ and $\phi_2$ are injective, $r^{-1}_{1}(\mc{F}_A)=\mr{Im}(\phi_1)$ and $r_{2}^{-1}(\mc{L}_A) =\mr{Im}(\phi_2)$.
Therefore, there exists $t \in H^1(\Hc_{X_{k}}(\det (\mc{F}_{k}),\det (\mc{F}_{k}))\otimes_k I)$ such that $\phi_2(t)=\mc{L}_{A'}$.
By Corollary \ref{d11}, $\mr{tr}^1\otimes \mr{Id}$ is surjective.
Hence there exists $t' \in H^1(\Hc_{X_{k}}(\mc{F}_{k},\mc{F}_{k})\otimes_k I)$ such that $\mr{tr}^1\otimes \mr{Id}(t')=t$. 
Denote by $\mc{F}_{A'}:=\phi_1(t')$. 
By commutativity of the lower left square, $\det(\mc{F}_{A'})=\mc{L}_{A'}$.
This concludes the proof of the theorem.
  \end{proof}

\section{Main results}\label{results}  
 
In Theorem \ref{artkn}, we showed that the deformation functor  $\mc{D}_{[\mc{F}_{k}],[\det (\mc{F}_{k})]}$ is unobstructed for any closed point $[\mc{F}_k]$ of the moduli space $M^{s}_{R, \mc{L}_{R}}$.
In this section we prove that this functor is in fact prorepresented by the completion of the local ring at the point $[\mc{F}_k]$ (see Proposition \ref{pro2}).
Using this we prove that the moduli space $M^{s}_{R, \mc{L}_{R}}$ of pure stable sheaves with fixed determinant $\mc{L}_R$ over $X_{R}$, is smooth over $\msp(R)$.

\begin{note}\label{nf} 
Keep Notations \ref{n} and \ref{nfd}.
Let $[\mc{F}_{k}]$ be a  $k$-rational point of $M^s_{R,\mc{L}_{R}}$ and denote by $\Lambda'':= \hat{\mo}_{M^{s}_{R,[\mc{F}_{k}]}}$, the completion of the local ring $\mo_{M^{s}_R,[\mc{F}_{k}]}$.
Under the determinant morphism $\mr{det}: M^s_R \to \mr{Pic}_{X_R}$, the line bundle $\mr{det}(\mc{F}_{k})$ is a $k$-point of $\mr{Pic}(X_R)$. 
Denote by $\Lambda':= \hat{\mo}_{\mr{Pic}(X_R),[\mr{det}(\mc{F}_{k})]}$ and by $\Lambda:= \hat{\mo}_{M^s_{R,\mc{L}_{R}},[\mc{F}_{k}]}$.
\end{note}
%Since $M^s_{R,\mc{L}_R}$ is the (scheme-theoretic) fiber under the determinant morphism over the $R$- point $[\mc{L}_{R}]$.

%Denote by $X_{\Lambda}:= X_{R}\times_{\msp(R)}\msp(\Lambda)$ and by $\mc{L}_{\Lambda}$ the pullback of $\mc{L}_{R}$ to $X_{\Lambda}$ under the natural morphism from $X_{\Lambda} \to X_{R}$. 

\begin{defi}
By ${\underline{\hat{\mo}}_{M^{s}_R,[\mc{F}_{k}]}}$ we denote the covariant functor 
\[ \Hom(\Lambda'',-): \mr{Art}/R \to \mr{Sets}, \ \ \  A \mapsto \Hom_{R-\mr{alg}}(\Lambda'',A) \]
We define the functors $\underline{\hat{\mo}}_{\mr{Pic}(X_R),[\mr{det}(\mc{F}_{k})]}$ and $\underline{\hat{\mo}}_{M^s_{R,\mc{L}_{R}},[\mc{F}_{k}]}$ similarly.
\end{defi}

\begin{lem}\label{pro1}
The deformation functor $\mc{D}_{[\mc{F}_{k}]}$ (resp. $\mc{D}_{[\mc{L}_{k}]}$) are pro-representable by $\hat{\mo}_{M^s_{R,[\mc{F}_{k}]}}$ (resp. $\hat{\mo}_{\mr{Pic}_{X_R},[\mr{det}(\mc{F}_{k})]}$).
\end{lem}

\begin{proof}
Recall from the proof of \cite[Theorem $3.1$]{EL}, that for $m$ sufficiently large, 
$\mc{R}^s$ is the open subset of $\mc{Q}uot(\mc{H};P)$ where $\mc{H}:=\mo_{X_R}(-m)^{P(m)}$ parametrizing stable quotients.
By \cite[Lemma $6.3$]{M}, $\phi:\mc{R}^s \to M_R^s$ is an etale $\mr{PGL}(V)$-principal bundle . 
Therefore, $\hat{\mo}_{\mc{R}^s,[\mc{F}_{k}]} \cong \hat{\mo}_{{M}^s_{R},[\mc{F}_{k}]}$. 
 
\vspace{0.2 cm} 
Denote by $Q:=\mc{Q}uot(\mc{H};P)$ and by $\mc{D}_{Q,[\mc{F}_{k}]}$ the deformation functor corresponding to the Quot-scheme at the point $[\mc{F}_{k}]$.
Recall that for any local Artin ring $A$, $\mr{Pic}(\msp(A)) = 0$,  hence $\mc{D}_{[\mc{F}_{k}]}=\mc{D}_{Q,[\mc{F}_{k}]}$. 
Since the functor $\mr{Quot}$ is representable, the deformation functor $\mc{D}_{Q,[\mc{F}_{k}]}$ is pro-representable by $\underline{\hat{\mo}}_{Q,[\mc{F}_{k}]}$ i.e., 
\[\mc{D}_{Q,[\mc{F}_{k}]} \cong \underline{\hat{\mo}}_{Q,[\mc{F}_{k}]} \cong \underline{\hat{\mo}}_{\mc{R}^s,[\mc{F}_{k}]},\]
where the second isomorphism follows from the fact that $\mc{R}^s$ is an open subset of $Q$.
Therefore, $\mc{D}_{[\mc{F}_{k}]}$ is isomorphic to $\underline{\hat{\mo}}_{M^s_{R},[\mc{F}_{k}]}$.

\vspace{0.4 cm}
Using the same argument we can show that $\mc{D}_{[\mr{det}(\mc{F}_{k})]} \cong \underline{\hat{\mo}}_{\mr{Pic}_{X_R},[\mr{det}(\mc{F}_{k})]}$.
This proves the lemma.
\end{proof}

Using this lemma we prove the following proposition.

\begin{prop} \label{pro2}
The deformation functor $\mc{D}_{[\mc{F}_{k}],[\mr{det}(\mc{F}_{k})]}$ is pro-represented by the completion of the local ring $\mc{O}_{M^{s}_{R,\mc{L}_R},[\mc{F}_{k}]}$. 
\end{prop}

\begin{proof}

By Lemma \ref{pro1}, $\mc{D}_{[\mc{F}_{k}]}$ (respectively $\mc{D}_{[\mr{det}(\mc{F}_{k})]}$) is pro-represented by $\hat{\mo}_{M^s_{R},[\mc{F}_{k}]}$ 
(respectively $\hat{\mo}_{\mr{Pic}(X_R),[\mr{det}(\mc{F}_{k})]}$). We have a natural transformation 
\[\underline{\mr{det}}:\underline{\hat{\mo}}_{M^s_{R,[\mc{F}_{k}]}} \to \underline{\hat{\mo}}_{\mr{Pic}(X_R),[\mr{det}(\mc{F}_{k})]}\] induced by the determinant morphism, $\mr{det}: M^{s}_R \to \mr{Pic}(X_{R})$
localized at the point $[F_{k}]$.
Let $A \in \mr{Art}/R$ and $\mc{L}_{A}$ be the pullback of the line bundle $\mc{L}_{R}$ under the morphism $X_A \to X_R$.
Recall the natural transformation ${\mc{D}et}_{A}$ defined in Definition \ref{ntdet}.
We have the following commutative diagram 

\[\begin{diagram}
   \mc{D}_{[\mc{F}_{k}]}(A)&\rTo^{\sim}&\underline{\hat{\mo}}_{M^s_{R,[\mc{F}_{k}]}}(A)\\
   \dTo^{\mc{D}et_{A}} &\circlearrowleft& \dTo^{\underline{\mr{det}}_{A}}\\
   \mc{D}_{[\mr{det}(\mc{F}_{k})]}(A)&\rTo^{\sim}_{\sigma}&\underline{\hat{\mo}}_{\mr{Pic}(X_R),[\mr{det}(\mc{F}_{k})]}(A)\\
  \end{diagram}\]
   
Hence the deformation functor  $\mc{D}_{[\mc{F}_{k}],[\mr{det}(\mc{F}_{k})]}(A) \cong {\underline{\mr{det}}_A}^{-1}(\phi_{\mc{L}_A})$, where $\phi_{\mc{L}_A}:= \sigma(\mc{L}_{A})$.
Therefore to prove that $\mc{D}_{[\mc{F}_{k}],[\mr{det}(\mc{F}_{k})]}$ is pro-represented by $\underline{\hat{\mo}}_{M^s_{R,\mc{L}_R},[\mc{F}_{k}]}$, we need to show that for any $A \in \mr{Art}/R$,

\begin{equation} \label{proeq}
 {\underline{\mr{det}}_A}^{-1}(\phi_{\mc{L}_A}) \cong  \mr{Hom}_{R}(\Lambda,A).  
\end{equation}

\noindent By Lemma \ref{pro1}, $\mc{D}_{[\mr{det}(\mc{F}_{k})]}(A) \xrightarrow{\sim} \Hom_{R}(\Lambda', A)$.
% By definition, $\mc{L}_{A}$ is the pull-back of $\mc{L}_{R}$ under the natural morphism $X_{A} \to X_{R}$ induced by $R \to A$.
% Since $\mc{D}_{[\mr{det}(\mc{F}_{k})]}$ is pro-representable, $\phi_{\mc{L}_A} \in \mr{Hom}(\Lambda, A)$ factors through $R$.
Hence for a fixed element $\mc{L}_A \in \mc{D}_{[\mr{det}(\mc{F}_{k})]}(A)$, the corresponding morphism from  $\msp(A) \to \msp(\Lambda')$ is unique and this is the morphism $\phi_{\mc{L}_{A}}$.
This implies the commutativity of the following diagram

\[\begin{diagram}
\msp(A) & \rTo & \msp(\Lambda'') \\ 
\dTo &\circlearrowleft& \dTo \\
\msp(R) & \rTo &  \msp(\Lambda')
\end{diagram}\]

\noindent where the morphism $\msp(R)  \to  \msp(\Lambda')$ is the morphism corresponding to the line bundle $\mc{L}_R$. 
Then the bijection in (\ref{proeq}) follows from the property of fibre product and the following diagram.

 
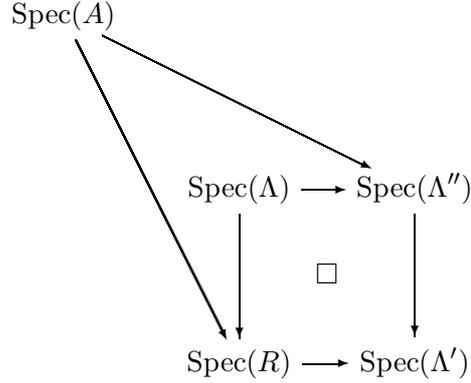
\[\begin{diagram}
 \msp(A)&&&&\\
 &\rdTo(4,2)\rdTo(2,4)&&&\\
 & &\msp(\Lambda) &\rTo&\msp(\Lambda'')\\
 & &\dTo&\square&\dTo\\
 & &\msp(R)&\rTo&\msp(\Lambda')\\
    \end{diagram}\]

\noindent Since $A$ was arbitrary, (\ref{proeq}) holds for any $A \in \mr{Art}/R$.  
%Hence, there exists a unique morphism from $\msp(A)$ to $\msp (\Lambda)$ making the above diagram commutative.
%Therefore ${\underline{\det}_A}^{-1}(\phi_{\mc{L}_A}) \simeq \underline{\hat{\mo}}_{M^s_{R,\mc{L}_{R}},[\mc{F}_{k}]}(A)$. 
Hence $\mc{D}_{[\mc{F}_{k}],[\mr{det}(\mc{F}_{k}]}$ is pro-represented by $\hat{\mc{O}}_{M^{s}_{R,\mc{L}_R},[\mc{F}_{k}]}$.
\end{proof}

\vspace{0.2 cm}

\noindent Using this we prove the following theorem.
 
 \vspace{0.2 cm}
 
\begin{thm}\label{smooth}
The morphism $M^{s}_{R,\mc{L}_{R}} \to \msp(R)$ is smooth. 
\end{thm}

\begin{proof}
Since the scheme $M^{s}_{R,\mc{L}_{R}}$ is noetherian and smoothness is an open condition, it suffices to check that the morphism $M^{s}_{R,\mc{L}_{R}} \to \msp(R)$ is smooth at closed points.
Let $[\mc{F}_{k}]$ be a closed point of $M^{s}_{R,\mc{L}_{R}}$.
Since the morphism $M^{s}_{R,\mc{L}_{R}} \to \msp(R)$ is of finite type, to prove that it is smooth at the point $[\mc{F}_{k}]$, we need to show that the functor $\underline{\hat{\mo}}_{M^s_{R,\mc{L}_{R}},[\mc{F}_{k}]}$ is unobstructed.

\vspace{0.2 cm}
By Proposition \ref{pro2}, the completion of the local ring $\mc{O}_{M^{s}_{R,\mc{L}_R},[\mc{F}_{k}]}$ pro-represents the functor $\mc{D}_{[\mc{F}_k],[\mr{det}(\mc{F}_k)]}$, i.e
$\underline{\hat{\mo}}_{M^s_{R,\mc{L}_{R}},[\mc{F}_{k}]} \simeq \mc{D}_{[\mc{F}_{k}],[\mr{det}(\mc{F}_{k}]}$.
By Theorem \ref{artkn}, the deformation functor $\mc{D}_{[\mc{F}_{k}],[\mr{det}(\mc{F}_{k})]}$ is unobstructed.
Hence the functor $\underline{\hat{\mc{O}}}_{M^{s}_{R,\mc{L}_R},[\mc{F}_{k}]} $ is unobstructed. 
This implies $\hat{\mc{O}}_{M^{s}_{R,\mc{L}_R},[\mc{F}_{k}]}$ is unobstructed.
Hence, the morphism  $M^s_{R,\mc{L}_{R}} \to \msp(R)$ is smooth at the point $[\mc{F}_k]$.
\end{proof}

\bibliographystyle{plain}
\bibliography{gen}

\begin{thebibliography}{1}

\bibitem{A}
V.~Artamkin.
\newblock On deformation of sheaves.
\newblock {\em Math USSR Izv}, 32:663--668, 1989.

\bibitem{EL}
H.~Esnault and A.~Langer.
\newblock On a positive equicharacteristic variant of the p-curvature
  conjecture.
\newblock {\em Documenta Math. J.}, 18:23--50, 2013.

\bibitem{FGA}
B.~Fantechi, L.~G{\"{o}}ttsche, L.~Illusie, S.~L. Kleiman, N.~Nitsure, and
  A.~Vistoli.
\newblock {\em Fundamental algebraic geometry. Grothendieck{'s} FGA
  explained,Mathematical Surveys and Monographs}, volume 123.
\newblock Amer. Math. Soc, 2005.

\bibitem{H}
R.~Hartshorne.
\newblock {\em Algebraic Geometry}, volume~52.
\newblock Graduate texts in Math, Springer Verlag, 1977.

\bibitem{HD}
R.~Hartshorne.
\newblock {\em Deformation Theory}, volume 257.
\newblock Graduate texts in Math, Springer Verlag, 2010.

\bibitem{HL}
D.~Huybrechts and M.~Lehn.
\newblock {\em The geometry of moduli spaces of sheaves}, volume~31.
\newblock Aspects of Mathematics, Vieweg, Braunshweig, 1997.

\bibitem{LA3}
A.~Langer.
\newblock Castenuovo-mumford regularity.
\newblock {\em Duke Math. J.}, 124:571--586, 2004.

\bibitem{LA1}
A.~Langer.
\newblock Semistable sheaves in positive characteristic.
\newblock {\em Ann of Math}, 159:251--276, 2004.

\bibitem{M}
M.~Maruyama.
\newblock Moduli of stable sheaves {II}.
\newblock {\em J.Math.Kyoto Univ}, 18:557--614, 1978.

\end{thebibliography}

\end{document}